\documentclass[11pt,twoside]{article}
\title{\bf Convergence Analysis of the Frank-Wolfe Algorithm and Its Generalization in Banach Spaces}
\author{
Hong-Kun Xu \\
{\small Department of Mathematics}\\
{\small Hangzhou Dianzi University } \\
{\small Hangzhou 310018} \\
{\small China} \\
{\small E-mail: xuhk@hdu.edu.cn}}

\date{}

\usepackage{amsfonts}
\usepackage[centertags]{amsmath}
\usepackage{amssymb}
\usepackage{amsthm}
\usepackage{cases}

 \newtheorem{theorem}{Theorem}[section]
 \newtheorem{cor}[theorem]{Corollary}
 \newtheorem{lemma}[theorem]{Lemma}
 \newtheorem{prop}[theorem]{Proposition}
 \theoremstyle{definition}
 \newtheorem{definition}[theorem]{Definition}
 \theoremstyle{remark}
 \newtheorem{remark}[theorem]{Remark}
 \theoremstyle{eg}
 \newtheorem{eg}[theorem]{Example}
 
 \theoremstyle{fact}
 
\numberwithin{equation}{section}

\begin{document}
\maketitle
\begin{abstract}

The Frank-Wolfe algorithm, a very first optimization method and
also known as the conditional gradient method, was introduced by Frank and Wolfe
in 1956. Due to its simple linear subproblems, the Frank-Wolfe algorithm has recently been
received much attention for solving large-scale structured optimization problems arising from
many applied areas such as signal processing and machine learning.
In this paper we will discuss in detail the convergence analysis of the Frank-Wolfe algorithm
 in Banach spaces.
Two ways of the selections of the stepsizes are discussed: the line minimization search method and the open loop rule.
In both cases, we prove the convergence of the Frank-Wolfe algorithm
in the case where the objective function $f$ has uniformly continuous (on bounded sets)
Fr\'echet derivative $f'$. We introduce the notion of the curvature constant of order $\sigma\in (1,2]$ and
obtain the rate $O(\frac{1}{k^{\sigma-1}})$ of convergence of the Frank-Wolfe algorithm.
In particular, this rate reduces to
$O(\frac{1}{k^{\nu}})$ if $f'$ is $\nu$-H\"older continuous for $\nu\in (0,1]$, and to
$O(\frac{1}{k})$ if $f'$ is Lipschitz continuous.
A generalized Frank-Wolfe algorithm
is also introduced to address the problem of minimizing a composite objective function.
Convergence of iterates of both Frank-Wolfe and generalized Frank-Wolfe algorithms are investigated. \\
\\



\end{abstract}


\section{Introduction}

The Frank-Wolfe algorithm (FWA) \cite{Frank-Wolfe1956}, a very first optimization method and also known as
the conditional gradient method \cite{Polyak1987}, was introduced by Frank and Wolfe  in 1956.
Consider a constrained convex minimization problem of the form:
\begin{equation}\label{min:0:f}
\min_{x\in C}  f(x),
\end{equation}
where $C$ is a nonempty compact convex subset of the Euclidean $d$-space $\mathbb{R}^d$
(with inner product $\langle\cdot,\cdot\rangle$ and norm $\|\cdot\|_2$) and
$f: \mathbb{R}^d\to \mathbb{R}$ is a differentiable, convex function.

Starting with an initial guess $x_0\in C$, FWA generates a sequence $\{x_k\}$
through the iteration process:
\begin{subequations} \label{al:FWA:0}
  \begin{numcases}{\hbox{}}
\label{FWA:0a}
 \bar{x}_k=\arg\min_{x\in C}\langle \nabla f(x_k),x\rangle,  \\
 \label{FWA:0b}
  x_{k+1}=x_k+\frac{2}{k+2}(\bar{x}_k-x_k).
 \end{numcases}
\end{subequations}
[Here $\nabla f$ is the gradient mapping of $f$.]
The idea of FWA is to approximate the objective function $f$ at the $k$th iterate $x_k$ by
its first-order expansion (i.e., linearization of $f$ at $x_k$) to get an intermediate point $\bar{x}_k$
via a linear minimization (\ref{FWA:0a})
in order to define the next iterate $x_{k+1}$ via a convex combination (\ref{FWA:0b}).
It is proved that $f(x_k)-f(x^*)\le O(\frac1{k})$ if $\nabla f$ is Lipschitz continuous, where
$x^*\in C$ is an optimal solution of (\ref{min:0:f}).

The gradient-projection algorithm (GPA) can also solve the minimization problem (\ref{min:0:f}).
GPA generates a sequence $\{x_k\}$ by the iteration process (\cite{Polyak1987,Xu2011}):
\begin{equation}\label{alg:pga}
x_{k+1}=P_C(x_k-\gamma_k\nabla f(x_k)),\quad k\ge 0,
\end{equation}
where $x_0\in C$ and $\{\gamma_k\}$ is a sequence of step-lengths.
Here $P_C$ is the projection operator onto $C$, that is,
$$P_Cx=\arg\min\{\|x-z\|_2: z\in C\},\quad x\in \mathbb{R}^d.$$

Therefore, FWA provides a projection-free algorithm for solving constrained optimization problems of
form (\ref{min:0:f}). Another feature of FWA is its simple linear subproblems, which is quite helpful in solving many large-scaled
optimization problems arising from applied areas such as signal/imaging processing and machine learning.
These make FWA revived recently in the study of optimization theory and methods \cite{Freund-Grigas2016,Ja2013,JH2014,JS2010}.
Early applications of FWA in the transportation theory may be found in \cite{Fu1984,V1987} and a recent decentralization
of FWA in network optimization may be found in \cite{WLSM2017}.

Now consider the constrained minimization problem (\ref{min:0:f}) in a Banach space $X$ with norm $\|\cdot\|$
and dual space $X^*$, and
$C$ a closed bounded convex subset of $X$. We point out that GPA (\ref{alg:pga}) is hardly extendable to
the Banach space framework since, on the one hand, the gradient of $f$, $\nabla f$, depends on the duality map
 $J: X\to X^*$ which is defined as
$$J(x)=\{x^*\in X^*: x^*(x)=\|x\|^2=\|x^*\|^2\},\quad x\in X.$$
Indeed, from Phelps \cite{Phelps1985},
$$\nabla f(x)=J^{-1}(f'(x)),\quad x\in X,$$
where $f'(x)$ is the Fr\'echet derivative of $f$ at $x$. Note that
the duality map $J$ is, in general, set-valued, and single-valued if and only if
the space $X$ is smooth (see \cite{Ci1990} for more connections of duality maps
with topological and geometrical properties of Banach spaces).

On the other hand, projections are not always well defined in a general Banach space.

In contrast with GPA, FWA has the advantage of involving with neither projections, nor duality maps. Therefore,
FWA can work in the Banach space setting.

This paper is aimed at studying the convergence and rate of convergence of FWA in a general Banach space $X$ for
solving the minimization problem (\ref{min:0:f}) and also the composite minimization problem
\begin{equation}\label{min:composite}
\min_{x\in C} \varphi(x):=f(x)+g(x),
\end{equation}
where $f, g\in\Gamma_0(X)$ are proper, lower semicontinuous, convex functions.

The main contributions of this paper are twofold:
\begin{itemize}
\item
Convergence of FWA for the minimization problems (\ref{min:0:f}) and (\ref{min:composite})
under two ways of selecting the stepsizes: line minimization search and open loop rule.
In this regard we assume that the Fr\'echet derivative $f'$ of $f$ be uniformly continuous over $C$,
which is weaker than the assumption in the literature that $f'$ be Lipschitz continuous.
\item
Rates of convergence of FWA under the above-mentioned two ways of choosing the stepsizes.
In this regard, we introduce the concept of curvature constant of order $\sigma\in (1,2]$ which
extends the notion of curvature constant \cite{JH2014} and which makes us able to obtain the
$O\left(\frac{1}{k^{\nu}}\right)$ rate of convergence of FWA in the case that $f'$ is $\nu$-H\"older
continuous, which is more general than the case of $f'$ being Lipschitz continuous in the literature.
\end{itemize}

The structure of the paper is as follows. In the next section we collect general notion and facts
of Fr\'echet derivatives, Lipschitz and H\"older continuity, and modulus of continuity.
We also include two lemmas which are main tools in proving convergence and rate of convergence
of the Frank-Wolfe algorithm and its generalization in subsequent sections.
In section \ref{sec:FWA}, we discuss convergence of FWA, including convergence of iterates
generated by FWA. In section \ref{sec:FWA-rate} we introduce the notion of constant curvature of
order $\sigma\in (1,2]$ which makes us able to obtain the convergence rate of FWA in the case
where the derivative $f'$ of $f$ is H\"older continuous (instead of Lipschitz continuous as
popularly used in current literature). This seems to be the first time in literature.
Section \ref{sec:gFWA} is devoted to an extension of FWA, known as generalized FWA, for solving composite
optimization problems of form (\ref{min:composite}). Many results of Sections \ref{sec:FWA}
and \ref{sec:FWA-rate} for FWA are extended to the generalized FWA for (\ref{min:composite}).
Finally, a summary of the results obtained in this paper is given in Section \ref{sec:conclusion}.

\section{Preliminaries}

Let $X$ be a Banach space with norm $\|\cdot\|$. Denote by $X^*$ the dual of $X$ and by
$\langle\cdot,\cdot\rangle$ the pairing between $X^*$ and $X$. Namely,
$$\langle x^*,x\rangle=x^*(x),\quad x^*\in X^*,\   x\in X.$$

A functional $f: X\to\mathbb{R}$ is said to be Fr\'echet differentiable at a point $x\in X$ if there
exists an element in $X^*$, denoted $f'(x)$, with the property
$$\lim_{u\to 0}\frac{f(x+u)-f(x)-\langle f'(x),u\rangle}{\|u\|}=0.$$
We say that $f$ is  Fr\'echet differentiable (on $X$) if $f$ is  Fr\'echet differentiable at every point $x\in X$.

Recall that a function $f: X\to \mathbb{R}$ is said to be
\begin{itemize}
\item
$L$-Lipschitz continuous for some $L>0$ if
$\|f(x)-f(y)\|\le L\|x-y\|$ for all $x,y\in X$;

\item
$\nu$-H\"older continuous for some $\nu\in (0,1]$ if there exists a constant $L_\nu>0$ such that
$\|f(x)-f(y)\|\le L_\nu\|x-y\|^\nu$ for all $x,y\in X$.
\end{itemize}
 For instance, if we define a function $h$ on $\ell^2$
by $h(x)=\|x\|_2^\sigma$ for $x\in \ell^2$ and $\sigma\in (1,2]$,
then the gradient of $h$, $\nabla h(x)=\sigma \|x\|_2^{\sigma-2}x$, is $(\sigma-1)$-H\"older continuous.

\begin{definition}
Let $X, Y$ be real Banach spaces and let $C$ be a nonempty subset of $X$.
The modulus of continuity of a function $f: C\to Y$ is defined by
\begin{equation*}
\omega(f,\tau):=\sup\{\|f(x_1)-f(x_2)\|_Y: x_1, x_2\in C,\   \|x_1-x_2\|\le\tau\}, \  \tau>0.
\end{equation*}
It is easily seen that $\omega(f,\tau)$ is a nondecreasing function of $\tau>0$. Moreover,
$f$ is uniformly continuous over $C$ if and only if $\lim_{\tau\to 0^+}\omega(f,\tau)=0$.
\end{definition}

The following result is straightforward and known.

\begin{prop}\label{prop:Holder}
Suppose $f$ is $\nu$-H\"older continuous for some $0<\nu\le 1$, namely,
\begin{equation}\label{holder}
\|f(x_1)-f(x_2)\|\le L_\nu\|x_1-x_2\|^\nu, \quad x_1, x_2\in C.
\end{equation}
Then $\omega(f,\tau)\le L_\nu\tau^\nu$ for $\tau>0$.
In particular, when $f$ is $L$-Lipschtz, namely,
\begin{equation}\label{LiL}
\|f(x_1)-f(x_2)\|\le L\|x_1-x_2\|, \quad x_1, x_2\in C,
\end{equation}
then $\omega(f,\tau)\le L\tau$ for $\tau>0$.
\end{prop}

To discuss the convergence of the FWA, we need the following lemma.

\begin{lemma} \cite{Xu2002}\label{le:convergence-tool}
Suppose a sequence $\{\alpha_k\}_{k=0}^\infty$ of nonnegative real numbers
satisfies the condition:
\begin{equation*}
\alpha_{k+1}\le (1-\eta_k)\alpha_k+\eta_k\varepsilon_k,\quad k\ge 0,
\end{equation*}
where $\{\eta_k\}$ and $\{\tau_k\}$ are sequences of nonnegative real numbers such that
\begin{enumerate}
\item[(a)] $\lim_{k\to\infty}\eta_k=0$;
\item[(b)] $\sum_{k=0}^\infty \eta_k=\infty$;
\item[(c)] $\lim_{k\to\infty}\varepsilon_k=0$.
\end{enumerate}
Then $\lim_{k\to\infty}\alpha_k=0$.
\end{lemma}

To obtain rate of convergence of FWA, we need the lemma below.

\begin{lemma} \cite[Lemma 6, page 46]{Polyak1987}\label{le:2:ineq}
Let $\{\alpha_k\}$ be a sequence of nonnegative real number satisfying the condition:
\begin{equation*}
\alpha_{k+1}\le\alpha_k-\beta_k\alpha_k^{1+\eta},\quad k\ge 0,
\end{equation*}
where $\beta_k\ge 0$ for all $k$, and $\eta>0$ is a constant. Then
$$\alpha_k\le\alpha_0\left(1+\eta\alpha_0^{\eta}\sum_{i=0}^{k-1}\beta_i\right)^{-\frac{1}{\eta}},\quad k\ge 1.$$
In particular,
\begin{itemize}
\item
if $\beta_k\equiv \beta$ for all $k$, then
\begin{equation*}
\alpha_{k}\le\frac{\alpha_0}{\left(1+\eta\alpha_0^{\eta}{\beta} k\right)^{\frac{1}{\eta}}},\quad k\ge 0;
\end{equation*}
\item
if $\beta_k\equiv \beta$ for all $k$ and $\eta=1$
(i.e., $\alpha_{k+1}\le\alpha_k-\beta\alpha_k^{2}$ for all $k$), then
$$\alpha_k\le \frac{\alpha_0}{1+\alpha_0{\beta}k},\quad k\ge 0.$$
\end{itemize}
\end{lemma}

\section{Convergence of the Frank-Wolfe Algorithm}
\label{sec:FWA}

Consider the minimization problem
\begin{equation}\label{min:f}
\min_{x\in C} f(x),
\end{equation}
where $C$ is a nonempty, closed, convex, bounded subset of a Banach space $X$,
and $f: X\to\mathbb{R}$ is a continuously Fr\'echet differentiable, convex function.
Assume (\ref{min:f}) has a nonempty set of solutions which is denoted by $S$.

Recall that the Frank-Wolfe algorithm (FWA)
generates a sequence $\{x_k\}$ by a two-stage iteration process as follows:
\begin{subequations} \label{al:FWA:1}
  \begin{numcases}{\hbox{}}
\label{FWA:1a}
 \bar{x}_k=\arg\min_{x\in C}\langle f'(x_k),x\rangle,  \\
 \label{FWA:1b}
  x_{k+1}=x_k+\gamma_k(\bar{x}_k-x_k).
 \end{numcases}
\end{subequations}
Here $\gamma_k\in [0,1)$ is the stepsize at the $k$th iteration.

\begin{remark}
FWA can be viewed as a fixed point algorithm. As a matter of fact, we have $x_{k+1}\in (1-\gamma_k)x_k+\gamma_kTx_k$, where
the (possibly set-valued) mapping $T$ is defined by
$$Tx:=\{z\in C: \langle f'(x),z\rangle=\inf_{w\in C} \langle f'(x),w\rangle\},\quad x\in C.$$
It is easily seen that $x\in C$ is a solution of (\ref{min:f}) if and only if $x\in C$ is a
fixed point of $T$, that is, $x\in Tx$.

Moreover, in order that the constrained linear minimization (\ref{FWA:1a}) be solvable for each $k$,
the set $C$ is required to be weakly compact. As a matter of fact,  since for each fixed $u^*\in X^*$, the linear
function $x\mapsto \langle u^*,x\rangle$ is weakly continuous, weak compactness of $C$ sufficiently implies that
the minimization $\min\{\langle u^*,x\rangle: x\in C\}$ has solutions.

Therefore, in what follows we actually implicitly assume that $C$ is  weakly compact convex, in particular,
$X$ is reflexive and $C$ is closed  bounded convex.

\end{remark}

The convergence of FWA (\ref{al:FWA:1}) depends on the choice of the stepsizes $\{\gamma_k\}$.
We will discuss in detail two ways of choosing the stepsizes $\{\gamma_k\}$: Line minimization search and open loop rule.

\subsection{Stepsizes by Line Minimization Search}

There are different ways of selecting the stepsizes
$\{\gamma_k\}$, one of which is the following one-dimensional line minimization search method:
\begin{equation}\label{gamma:1}
\gamma_k=\arg\min_{0\le \gamma\le 1}f(x_k+\gamma(\bar{x}_k-x_k)).
\end{equation}
Note that the first-order approximation at $x_k$ to $f$ is the linear function:
$$f_k(x):=f(x_k)+\langle f'(x_k),x-x_k\rangle.$$
An equivalent definition of $\bar{x}_k$ is thus given by
$\bar{x}_k=\arg\min_{x\in C} f_k(x).$
Note also that $f'_k(x)=f'(x_k)$ for all $x$. Consequently, another equivalent
condition for $\bar{x}_k$ is the variational inequality (VI):
\begin{equation}\label{vi:1}
\langle f'(x_k),x-\bar{x}_k\rangle\ge 0,\quad x\in C.
\end{equation}

The result below was proved in Polyak \cite{Polyak1987} in a Hilbert space and under the condition
that the Fr\'echet derivative $f'$ of $f$ is Lipschitz continuous on $C$.
Here we prove, in a different argument from Polyak's, the same result
in a Banach space and under the weaker condition that $f'$ be uniformly continuous on $C$.

\begin{theorem}\label{th:FWA-convergence1}
Let $C$ be a closed bounded convex subset of a real Banach space $X$
and let $f: X\to\mathbb{R}$ be a differentiable convex function such that
the Fr\'echet derivative $f'$ is uniformly continuous on $C$.
Let $\{x_k\}$ be generated by FWA (\ref{al:FWA:1}),
where the sequence of stepsizes, $\{\gamma_k\}$, is selected by the
line minimization search method (\ref{gamma:1}). Then
\begin{itemize}
\item[(i)]
$f(x_{k+1})\le f(x_k)$ for all $k$, and
\item[(ii)]
$\lim_{k\to\infty} f(x_k)=f^*$.
\end{itemize}

\end{theorem}
\begin{proof}
Put $\theta_k=f(x_k)-f^*$ for $k\ge 0$, and define a function $g_k(\gamma)$ by
$$g_k(\gamma)=f(x_k+\gamma(\bar{x}_k-x_k)),\quad 0\le \gamma\le 1.$$
Then $f(x_{k+1})=\min\{g_k(\gamma): 0\le \gamma\le 1\}\le g_k(0)=f(x_k)$.
This proves (i).

To see (ii), we take a null sequence $\{\tau_k\}$ in $(0,1)$ such that $\sum_{k=0}^\infty\tau_k=\infty$
to deduce that
\begin{align}\label{fx0}
f(x_{k+1})&\le g_k(\tau_k)=f(x_k+\tau_k(\bar{x}_k-x_k)) \nonumber \\
&=f(x_k)+\int_0^1 \langle f'(x_k+t\tau_k(\bar{x}_k-x_k)),\tau_k(\bar{x}_k-x_k)\rangle dt \nonumber \\
&=f(x_k)+\tau_k\langle f'(x_k),\bar{x}_k-x_k\rangle \nonumber \\
&\quad +\tau_k\int_0^1 \langle f'(x_k+t\tau_k(\bar{x}_k-x_k))-f'(x_k),\bar{x}_k-x_k\rangle dt.
\end{align}
Set $\delta={\rm diam}(C)$ and
$\varepsilon_k=\delta\cdot\sup_{0\le t\le 1}\|f'(x_k+t\tau_k(\bar{x}_k-x_k))-f'(x_k)\|$.
Then (\ref{fx0}) is reduced to the inequality
\begin{equation}\label{fx1}
f(x_{k+1})\le f(x_k)+\tau_k\langle f'(x_k),\bar{x}_k-x_k\rangle+\tau_k\varepsilon_k.
\end{equation}
On the other hand, the convexity of $f$ implies that, for any $x\in C$,
$$f(x)\ge f(x_k)+\langle f'(x_k),x-x_k\rangle\ge f(x_k)+\langle f'(x_k),\bar{x}_k-x_k\rangle.$$
Consequently,
\begin{equation}\label{f'xk0}
\langle f'(x_k),\bar{x}_k-x_k\rangle\le f^*-f(x_k)=-\theta_k.
\end{equation}
Substituting (\ref{f'xk0}) into (\ref{fx1}), we get
\begin{equation}\label{fx2}
\theta_{k+1}\le (1-\tau_k)\theta_k+\tau_k\varepsilon_k.
\end{equation}
Since $f'$ is uniformly continuous over $C$ and since $\|t\tau_k(\bar{x}_k-x_k)\|\le\delta\tau_k\to 0$
as $k\to\infty$, we obtain $\varepsilon_k\to 0$ as $k\to\infty$.

Now applying Lemma \ref{le:convergence-tool} to (\ref{fx2}),
we conclude that $\theta_k\to 0$ as $k\to\infty$.
\end{proof}

\subsection{Stepsizes by Open Loop Rule}

The open loop rule was introduced in \cite{DH1978} to study convergence of FWA.
This rule means that the sequence $\{\gamma_k\}$ of stepsizes satisfies the following
two conditions:
\begin{enumerate}
\item[(C1)]
$\lim_{k\to\infty}\gamma_k=0$,
\item[(C2)] $\sum_{k=0}^\infty\gamma_k=\infty$.
\end{enumerate}

\begin{theorem}\label{th:open-loop-1}
Let $C$ be a nonempty closed bounded convex subset of a real Banach space $X$, let $f: X\to\mathbb{R}$
be a convex function with a uniformly continuous Fr\'echet derivative $f'$ over $C$, and let
$\{x_k\}$ be generated by the Frank-Wolfe algorithm (\ref{al:FWA:1}).
Suppose $\{\gamma_k\}\subset (0,1]$ satisfies the open loop rule (C1)-(C2).
Then $\lim_{k\to\infty}f(x_k)=\inf_C f$.
\end{theorem}
\begin{proof}
Recall that we have
$$x_{k+1}=x_k+\gamma_k(\bar{x}_k-x_k),\quad \bar{x}_k=\arg\min_{x\in C}\langle f'(x_k),x\rangle.$$
Put again
$\theta_k=f(x_k)-f^*$.
We now have
\begin{align}\label{fxk+1}
f(x_{k+1})&=f(x_k)+\int_0^1 \langle f'(x_k+t(x_{k+1}-x_k)),x_{k+1}-x_k\rangle dt \nonumber \\
&=f(x_k)+\gamma_k\langle f'(x_k),\bar{x}_k-x_k\rangle \nonumber \\
&\quad +\gamma_k\int_0^1 \langle f'(x_k+t\gamma_k(\bar{x}_{k}-x_k))-f'(x_k),\bar{x}_{k}-x_k\rangle dt.
\end{align}

Since $f$ is convex, we get, for any $x\in C$,
$$f(x)\ge f(x_k)+\langle f'(x_k),x-x_k\rangle\ge f(x_k)+\langle f'(x_k),\bar{x}_k-x_k\rangle.$$
It turns out that
\begin{equation}\label{f'xk}
\langle f'(x_k),\bar{x}_k-x_k\rangle\le f^*-f(x_k)=-\theta_k.
\end{equation}
On the other hand, since $\|x_k+t\gamma_k(\bar{x}_{k}-x_k)-x_k\|\le \gamma_k\delta\to 0$,
the uniform continuity of $f'(x)$ over $x\in C$ results that
\begin{align}\label{int}
&\int_0^1 \langle f'(x_k+t\gamma_k(\bar{x}_{k}-x_k))-f'(x_k),\bar{x}_{k}-x_k\rangle dt \nonumber\\
&\qquad\le \delta\cdot\sup_{0\le t\le 1} \|f'(x_k+t\gamma_k(\bar{x}_{k}-x_k))-f'(x_k)\|=:\varepsilon_k\to 0
\end{align}
since $\|t\gamma_k(\bar{x}_{k}-x_k)\|\le\gamma_k\delta\to 0$.

Substituting (\ref{f'xk}) and (\ref{int}) into (\ref{fxk+1}) yields
\begin{equation}\label{theta}
\theta_{k+1}\le (1-\gamma_k)\theta_k+\gamma_k\varepsilon_k.
\end{equation}
Now applying Lemma \ref{le:convergence-tool}, we obtain $\theta_k\to 0$.
\end{proof}

\begin{remark}
In \cite[Theorem 1]{DH1978}, Dunn and Harshbarger assumed that $\{\gamma_k\}\subset (0,1]$ satisfies the conditions:\\

(DH1)  $\gamma_k\le \frac{\alpha}{k}$ for some constant $\alpha>0$ and all $k\ge 1$, and

(DH2)  $1-\gamma_{k+1}=\frac{\gamma_{k+1}}{\gamma_k}$ or equivalently,
$\gamma_{k+1}=\frac{\gamma_{k}}{1+\gamma_k}$ for all $k\ge 0$.\\

It is not hard to see that conditions (DH1)-(DH2) imply (C1)-(C2). In fact, by induction, it is easy to see
$\gamma_k\ge\frac{\gamma_0}{k+1}$ for all $k\ge 0$. This is trivial for $k=0$.
Suppose this is true for some $k>0$. Then, we infer that (noting $\gamma_0\le 1$)
$$\gamma_{k+1}=\frac{\gamma_k}{1+\gamma_k}\ge\frac{\frac{\gamma_0}{1+k}}{1+\frac{\gamma_0}{1+k}}
=\frac{\gamma_0}{1+k+\gamma_0}\ge\frac{\gamma_0}{2+k}.$$
Consequently, $\sum_{k=0}^\infty\gamma_k=\infty$.

We also find that (DH2) implies (DH1). Indeed, (DH2) implies
\begin{equation}\label{eq:gammaD2}
\gamma_k\le\frac{1}{k+\frac1{\gamma_0}}=\frac{\gamma_0}{\gamma_0 k+1},\quad k\ge 0.
\end{equation}
It is clear that (\ref{eq:gammaD2}) holds when $k=0$.
Assume (\ref{eq:gammaD2}) holds for some $k>0$. We then get by (DH2)
\begin{align*}
\gamma_{k+1}&=\frac{\gamma_{k}}{1+\gamma_k}\\
&\le \frac{\frac{\gamma_0}{\gamma_0 k+1}}{1+\frac{\gamma_0}{\gamma_0 k+1}}
=\frac{\gamma_0}{\gamma_0 k+1+\gamma_0}\\
&=\frac{\gamma_0}{\gamma_0 (k+1)+1}=\frac{1}{k+1+\frac1{\gamma_0}}.
\end{align*}
Therefore, (DH1) holds for all $k\ge 1$ with $\alpha=1$.
\end{remark}

\begin{remark}
Theorems \ref{th:FWA-convergence1}and \ref{th:open-loop-1} show that in a finite-dimensional space,
FWA \eqref{al:FWA:1} converges under the condition that the gradient $\nabla f$ of $f$ is continuous on $C$.
This is sharp in the sense that FWA (\ref{al:FWA:1}) may fail to converge if $\nabla f$ is discontinuous,
as shown by the following example of Nesterov.
\begin{eg}
\cite[Example 1]{Nesterov2017} Consider $X=\mathbb{R}^2$ equipped with the Euclidean norm $\|\cdot\|_2$,
$C=\{x=(x^{(1)},x^{(2)})^\top\in\mathbb{R}^2: (x^{(1)})^2+(x^{(2)})^2\le 1\}$ is the closed unit disc, and
$f(x)=\max\{x^{(1)},x^{(2)}\}$ for $x\in\mathbb{R}^2$.
Then $f$ is nondifferentiable for $x^{(1)}=x^{(2)}$, and differentiable for $x^{(1)}\not=x^{(2)}$ with
$\nabla f(x)=(0,1)^\top$ if $x^{(1)}<x^{(2)}$, and $(1,0)^\top$ if $x^{(1)}>x^{(2)}$.
It is also not hard to find that the unique minimizer of $f$ over $C$ is
$x^*=-(\frac{1}{\sqrt{2}},\frac{1}{\sqrt{2}})^\top$. Moreover, starting with any initial
$x_0\not=x^*$, the sequence $\{x_k\}$ generated by FWA  (\ref{al:FWA:1}) fully lies in the triangle
with vertices $\{x_0, (-1,0)^\top, (0,-1)^\top\}$. It turns out that $\{f(x_k)\}$ fails to converge to
the optimal value of $f$ over $C$.
\end{eg}
\end{remark}

\subsection{Convergence of Iterates}

We now discuss the convergence of the iterates $\{x_k\}$ generated by the FWA (\ref{al:FWA:1}).
We will assume that the space $X$ is reflexive so that every bounded convex subset of $X$
is weakly compact. Recall that a function $h$ is said to be uniformly convex if there exists a
continuous function $\delta: [0,\infty)\to [0,\infty)$,
$\delta(0)=0$ and $\delta(t)>0$ for all $t>0$, such that
\begin{equation}\label{eq:h}
h(\lambda x+(1-\lambda)y)\le \lambda h(x)+(1-\lambda)h(y)-\lambda(1-\lambda)\delta(\|x-y\|)
\end{equation}
for all $\lambda\in (0,1)$ and $x,y\in X$. We will call $\delta$ a modulus of convexity of $h$.
In particular, when $\delta(t)=ct^2$ for some constant $c>0$, $h$ is called strongly convex.

Theorems \ref{th:FWA-convergence1} and \ref{th:open-loop-1} imply that each weak cluster point
$x^*$ of the iterates $\{x_k\}$ is an optimal solution of (\ref{min:f}).
An interesting and natural question is whether the full sequence $\{x_k\}$ converges weakly.
The following result is a partial answer to this question.

\begin{theorem}\label{th:FWA-iterates}
Let $X$ be a reflexive Banach space and consider
the FWA (\ref{al:FWA:1}) with the stepsize sequence $\{\gamma_k\}$ selected by either the line minimization
search method in Theorem \ref{th:FWA-convergence1} or the
open loop rule (C1)-(C2) of Theorem \ref{th:open-loop-1}.
\begin{enumerate}
\item[{\rm (i)}]
If $f$ is strictly convex, then  $\{x_k\}$ converges weakly to the unique solution of (\ref{min:f}).
\item[{\rm (ii)}]
If $f$ is uniformly convex, then $\{x_k\}$ converges in norm to the unique solution of (\ref{min:f}).
\item[{\rm (iii)}]
If $f$ has a sharp minimum point $x^*$, then $\{x_k\}$ converges in norm to $x^*$ at a finite termination.
\item[{\rm (iv)}]
If $C$ is compact in the norm topology, if the stepsizes $\{\gamma_k\}$ is selected by the open loop rule,
and if $\{x_k\}$ has at most finitely many cluster points, then
$\{x_k\}$ converges in norm to a solution of (\ref{min:f}).
\end{enumerate}

\end{theorem}
\begin{proof}
(i)
In this case, $f$ has a unique minimum in $C$ which we denote by $x^*$.
By Theorems \ref{th:FWA-convergence1} and \ref{th:open-loop-1}, we know that every weak
cluster point of $\{x_k\}$  is a minimum of $f$. By uniqueness of minimum of $f$,
we find that $\{x_k\}$ has one (note that $C$ is weakly compact) and only one weak
cluster point, hence, must be convergent weakly to $x^*$.

(ii)
First observe by (i) that $\{x_k\}$ is weakly convergent to the unique solution $x^*$ of (\ref{min:f}).
Now let $\delta$ be a modulus of convexity of $f$ (i.e., Eq. (\ref{eq:h}) holds for $f$).
It turns out that
$$\frac{f(y+\lambda(x-y))-f(y)}{\lambda}\le f(x)-f(y)-(1-\lambda)\delta(\|x-y\|).$$
Letting $\lambda\to 0$ yields
\begin{equation}\label{eq:3:f3}
f(x)\ge f(y)+\langle f'(y),x-y\rangle+\delta(\|x-y\|)
\end{equation}
for all $x,y\in X$. In particular, taking $x:=x_k$ and $y:=x^*\in S$ implies that
\begin{equation}\label{eq:3:converge}
f(x_k)\ge f(x^*)+\langle f'(x^*),x_k-x^*\rangle+\delta(\|x_k-x^*\|).
\end{equation}
Since $f(x_k)\to f(x^*)$ and $x_k\to x^*$ weakly, taking the limit in (\ref{eq:3:converge}) as $k\to\infty$,
we immediately get $\delta(\|x_k-x^*\|)\to 0$. Consequently, $x_k\to x^*$ in norm.

(iii)
Recall that the definition of $f$ having a sharp minimum point $x^*$
means that there exists $\alpha>0$ such that \cite[page 136]{Polyak1987}
\begin{equation}\label{sharp}
f(x)\ge f(x^*)+\alpha\|x-x^*\|
\end{equation}
for all $x\in C$. It then turns out that
$$\|x_k-x^*\|\le\frac{1}{\alpha}[f(x_k)-f^*]\to 0.$$
That is, $x_k\to x^*$ in norm; hence, $\|f'(x_k)-f'(x^*)\|_*\to 0$ as well.

Observe that, in this case, $\bar{x}_k$ is the unique solution to VI (\ref{vi:1}).
However, noting that (\ref{sharp}) implies that
$$\langle f'(x^*),x-x^*\rangle\ge\alpha\|x-x^*\|,\quad x\in C, $$
we obtain
\begin{align*}
\langle f'(x_k),x-x^*\rangle
&=\langle f'(x_k)-f'(x^*),x-x^*\rangle+\langle f'(x^*),x-x^*\rangle  \\
&\ge -\|f'(x_k)-f'(x^*)\|_*\|x-x^*\|+\alpha\|x-x^*\|\\
&=\|x-x^*\|(\alpha-\|f'(x_k)-f'(x^*)\|_*)\ge 0
\end{align*}
for all $k$ large enough so that $\|f'(x_k)-f'(x^*)\|_*<\alpha$.
For any such $k$, we find that $x^*$ is also a solution of VI (\ref{vi:1}) and thus
$\bar{x}_k=x^*$ by uniqueness, which implies that $\gamma_k=1$ and $x_{k+1}=x^*$.

(iv)
Since $C$ is compact in the norm topology, $\{x_k\}$ is relatively compact in the strong topology.
Hence the set of strong cluster points of $\{x_k\}$ is nonempty. Denote this set by $\Omega$.
We must verify that $\Omega$ is singleton. By assumption we know that $\Omega$ is a finite set, which is enumerated
as $\Omega=\{x_1^*,\cdots,x_m^*\}$, where $m\ge 1$ is an integer. We next prove $m=1$ by contradiction.
Suppose on the contrary that $m>1$.
Let $\varepsilon$ satisfy
\begin{equation}\label{eq:3:epsilon}
0<\varepsilon<\frac{\min\{\|x^*_i-x^*_j\|: 1\le i\not=j\le m\}}{\max\{m+1,3\}}
\end{equation}
and define
$$N_i:=\{k\in\mathbb{N}: \|x_k-x^*_i\|<\varepsilon\},\quad i=1,2,\cdots, m.$$
Then $\{N_i\}$ are mutually disjoint: $N_i\cap N_j=\emptyset$ for all $i\not=j$. Moreover,
$$\mathbb{N}\setminus \cup_{i=1}^m N_i$$
is at most a finite set. Therefore, we may assume that
$$\mathbb{N}=\cup_{i=1}^m N_i.$$
Now by (C1) (i.e., $\gamma_k\to 0$), we find from (\ref{FWA:1b}) that
$\|x_{k+1}-x_k\|\to 0$. Let $k_0$ satisfy
$$\|x_{k+1}-x_k\|<\varepsilon$$
for all $k\ge k_0$. Now let $k'>k_0$ be the smallest integer such that
\begin{equation}\label{eq:3:x1}
\|x_{k'}-x^*_1\|<\varepsilon.
\end{equation}
Namely, $k'\in N_1$.
Consequently, $k'-1\in N_{i'}$ for some $i'>1$ (it is impossible that $k'-1\in N_{1}$
by virtue of (\ref{eq:3:epsilon})). We now arrive at the contradiction:
\begin{align*}
3\varepsilon<\|x^*_1-x^*_{i'}\|\le\|x^*_1-x_{k'}\|+\|x_{k'}-x_{k'-1}\|+\|x_{k'-1}-x^*_{i'}\|<3\varepsilon.
\end{align*}
This finishes the proof of (iv).
\end{proof}

\begin{remark}
Part (iii) of Theorem \ref{th:FWA-iterates} is also proved in \cite[Theorem 3, p. 211]{Polyak1987}
in a Hilbert space and under the assumption that $\nabla f$ be Lipschitz continuous.

\end{remark}

\section{Rate of Convergence of the Frank-Wolfe Algorithm}
\label{sec:FWA-rate}

The concept of curvature constant plays a key role in discussing the rate of convergence of FWA.

\begin{definition}\label{def:curvature1}\cite{Ja2013,JH2014,JS2010}
Let $C$ be a nonempty closed convex bounded subset of a real Banach space $\mathbb{R}^d$
and let $f: \mathbb{R}^d\to\mathbb{R}$ be a differentiable  function.
The curvature constant of $f$ over $C$, $C_f$, is defined as the number in $[0,\infty]$:
\begin{equation}\label{curvature-C}
C_f=\sup_{{x,s\in C \atop \gamma\in (0,1]}\atop y=x+\gamma(s-x)}
\frac{2}{\gamma^2}(f(y)-f(x)-\langle y-x,\nabla f(x)\rangle).
\end{equation}
\end{definition}
The curvature constant $C_f$ plays a key role in the analysis of convergence rate of FWA
in the case where $f$ has a Lipschitz continuous gradient, as shown in the result below.

\begin{theorem}\label{th:FWA-rate0}\cite[Theorem 1]{Ja2013,JH2014}
Let $\{x_k\}$ be generated by FWA (\ref{al:FWA:0}). Then
$$f(x_k)-f(x^*)\le \frac{2C_f}{k+2},$$
where $x^*\in C$ is an optimal solution of (\ref{min:0:f}).
\end{theorem}

\begin{remark}
We notice that the notion of curvature constant works for the case where $\nabla f$
is Lipschitz continuous. As a matter of fact, if $\nabla f$ is $L$-Lipschitz, then it is
easy to find that $C_f\le \delta^2L$, where $\delta:=\sup\{\|u-v\|: u,v\in C\}<\infty$ is
diameter of $C$. However, it does not work for the situation where $f'$ fails to be
Lipschitz continuous, for instance, $f'$ being
$\nu$-H\"older continuous for $\nu\in (0,1)$, as shown by the following simple example.

\begin{eg}\label{eq:curvature}
Take $f(t)=t^\alpha$, $t\in \mathbb{R}$, $\alpha\in (1,2)$, and
$C=[0,1]$. Then $\nabla f(t)=\alpha t^{\alpha-1}$ is $(\alpha-1)$-H\"older continuous
(not Lipschitz continuous). It is easily found that $C_f=\infty$.
\end{eg}

\subsection{Curvature Constant of Order $\sigma$}

In order to accommodate the case where the gradient $\nabla f$ is non-Lipschitz continuous, we
here introduce the notion of curvature constant of order $\sigma$ of $f$ over $C$.

\end{remark}

\begin{definition}\label{def:curvature2}
Let $C$ be a nonempty closed bounded convex subset of a real Banach space $X$
and let $f: X\to\mathbb{R}$ be a differentiable  function.
The curvature constant of $f$ of order $\sigma\in (1,2]$ over $C$, $C_f^\sigma$, is defined as the number:
\begin{equation}\label{curvature-order1}
C^{(\sigma)}_f=\sup_{{x,s\in C \atop \gamma\in (0,1]}\atop y=x+\gamma(s-x)}
\frac{\sigma}{\gamma^\sigma}(f(y)-f(x)-\langle y-x,f'(x)\rangle).
\end{equation}
Equivalently, $C^{(\sigma)}_f\ge 0$ is the least nonnegative number such that
\begin{equation}\label{curvature-order2}
f(y)\le f(x)+\langle y-x,f'(x)\rangle+\frac{\gamma^{\sigma}}{\sigma}C^{(\sigma)}_f
\end{equation}
for all $x,y\in C$ such that $y=x+\gamma(s-x)$ for all $0\le\gamma\le 1$ and $s\in C$.

When $\sigma=2$, $C^{(2)}_f$ coincides with the curvature constant $C_f$ of Definition \ref{def:curvature1}.
\end{definition}

It is not hard to find that the curvature constant of order $\alpha\in (1,2)$ of the function $f$ over $[0,1]$
defined in Example \ref{eq:curvature} is $C_f^{(\alpha)}=\alpha$ (recall $C_f=\infty$).
Indeed, since $f'$ is $(\alpha-1)$-H\"older continuous with constant $L_\alpha=\alpha$,
we have by Corollary \ref{cor:curva-Holder} below that $C_f^{(\alpha)}\le\alpha$.
On the other hand,
\begin{align*}
C_{f}^{(\alpha)}&=\sup_{{x,s,\gamma\in (0,1)}\atop y=x+\gamma(s-x)}
\frac{\alpha}{\gamma^\alpha}(y^\alpha-x^\alpha-\alpha(y-x)x^{\alpha-1})\\
&=\alpha\cdot
\sup_{\gamma, s,x\in (0,1)}\left\{\left(\frac{x}{\gamma}+s-x\right)^\alpha-\left(\frac{x}{\gamma}\right)^\alpha
-\alpha(s-x)\left(\frac{x}{\gamma}\right)^{\alpha-1}\right\}.
\end{align*}
Taking $x=0$ immediately implies that $C_{f}^{(\alpha)}\ge\alpha$; hence $C_{f}^{(\alpha)}=\alpha$.

\begin{remark}
Assume $f$ is continuously Fr\'echet differentiable and strongly convex with power $\sigma\in (1,2]$, namely, 
\begin{equation}\label{eq:convex-power}
f(\lambda x+(1-\lambda)y)\le \lambda f(x)+(1-\lambda)f(y)-\mu_\sigma W_\sigma(\lambda)\|x-y\|^\sigma
\end{equation}
for all $\lambda\in (0,1)$ and $x,y\in X$, where $\mu_\sigma>0$ is a constant and 
$W_\sigma(\lambda)=\lambda^\sigma(1-\lambda)+\lambda(1-\lambda)^\sigma$.
Note that (\ref{eq:convex-power}) implies that
\begin{equation*}
f(x)\ge f(y)+\langle f'(y),x-y\rangle+\mu_\sigma\|x-y\|^\sigma,\quad x,y\in X.
\end{equation*}
As a result, we obtain a lower bound for the curvature constant of order $\sigma$ as follows:
$$C_f^{(\sigma)}\ge\frac{\sigma}{\mu_\sigma}({\rm diam}(C))^\sigma.$$
In particular, if $f$ is strongly convex (i.e., strongly convex with power 2), then 
we have a lower bound for the curvature constant:
$$C_f\ge\frac{2}{\mu_2}({\rm diam}(C))^2.$$

\end{remark}

We can use the modulus of continuity of the Fr\'echet derivative $f'$ to
estimate $C^{(\sigma)}_f$.

\begin{prop}\label{prop:modulus}
Suppose $f'$ is uniformly continuous over $C$. Then the
curvature constant of order $\sigma$ of $f$ has the estimate:
\begin{equation}\label{eq:C-f}
C^{(\sigma)}_f\le\sup_{0<\gamma\le 1}\frac{\sigma}{\gamma^\sigma}\int_0^{\gamma\cdot{\rm diam}(C)}\omega(f',\tau)d\tau.
\end{equation}

\end{prop}
\begin{proof}
Since (recalling $y=x+\gamma(s-x)$)
\begin{align*}
f(y)&=f(x)+\int_0^1 \langle f'(x+t(y-x)),y-x\rangle dt\\
&=f(x)+\langle f'(x),y-x\rangle+\int_0^1 \langle f'(x+t(y-x))-f'(x),y-x\rangle dt\\
&=f(x)+\langle f'(x),y-x\rangle+\gamma\int_0^1 \langle f'(x+t\gamma(s-x))-f'(x),s-x\rangle dt,
\end{align*}
it turns out that
\begin{align*}
C^{(\sigma)}_f&=\sup_{{x,s\in C \atop \gamma\in [0,1]}\atop y=x+\gamma(s-x)}
\frac{\sigma}{\gamma^\sigma}(f(y)-f(x)-\langle y-x,f'(x)\rangle)\\
&=\sup_{x,s\in C\atop 0<\gamma\le 1}\frac{\sigma}{\gamma^\sigma}\cdot\gamma\int_0^1
\langle f'(x+t\gamma(s-x))-f'(x),s-x\rangle dt\\
&\le\sup_{x,s\in C\atop 0<\gamma\le 1}\frac{\sigma}{\gamma^\sigma}\cdot \gamma
\int_0^1 \omega(f',t\gamma\|s-x\|)\|s-x\| dt\\
&=\sup_{x,s\in C\atop 0<\gamma\le 1}\frac{\sigma}{\gamma^\sigma}\int_0^{\gamma\|s-x\|} \omega(f',\tau)d\tau\\
&=\sup_{0<\gamma\le 1}\frac{\sigma}{\gamma^\sigma}\int_0^{\gamma\cdot{\rm diam}(C)} \omega(f',\tau)d\tau.
\end{align*}
\end{proof}
\begin{cor}\label{cor:curva-Holder}
If $f'$ is $\nu$-H\"older continuous for some $0<\nu\le 1$ with constant $L_\nu$,
 then $C^{(1+\nu)}_f\le L_{\nu}\delta^{1+\nu}$. In particular, if $f'$ is $L$-Lipschitz, then
$C_f\le L \delta^{2}$. (Here $\delta={\rm diam}(C)$.)
\end{cor}
\begin{proof}
By Proposition \ref{prop:modulus}, we get (with $\sigma=1+\nu$)
\begin{align*}
C^{(\sigma)}_f
&\le\sup_{0<\gamma\le 1}\frac{\sigma}{\gamma^\sigma}\int_0^{\gamma d} \omega(f',\tau)d\tau\\
&\le\sup_{0<\gamma\le 1}\frac{\sigma}{\gamma^\sigma}\int_0^{\gamma d} L_{\nu}\tau^{\nu}d\tau\\
&=\sup_{0<\gamma\le 1}\frac{\sigma}{\gamma^\sigma}\frac{L_{\nu}\gamma^{1+\nu}}{1+\nu}d^{1+\nu}\\
&=L_{\nu}\delta^{1+\nu}.
\end{align*}
\end{proof}

\subsection{Stepsizes by Line Minimization}

We can now show the role played by the curvature constant of order $\sigma$ in
the analysis of rate of convergence of FWA.

\begin{theorem}\label{th:FWA-rate1}
Let the assumptions in Theorem \ref{th:FWA-convergence1} hold.
Assume, in addition, that there exists $\sigma>1$ such that
the curvature constant of order $\sigma$ of $f$ over $C$, $C_f^{(\sigma)}$, is finite.
Let $\{x_k\}$ be generated by the Frank-Wolfe algorithm (\ref{al:FWA:1}),
where the sequence of stepsizes, $\{\gamma_k\}$, is selected by the
line minimization search method (\ref{gamma:1}). Then we have
\begin{equation}
f(x_k)-f^*\le \frac{\theta}
{\left(1+\frac{1}{\sigma}\theta^{\frac{1}{\sigma-1}}(C^{(\sigma)}_f)^{\frac{1}{1-\sigma}}\cdot k\right)^{\sigma-1}}
=O\left(\frac{1}{k^{\sigma-1}}\right),
\end{equation}
where $\theta=f(x_0)-f^*$. In particular, we get
\begin{itemize}
\item
If $f'$ is $\nu$-H\"older continuous with constant $L_\nu$, then
\begin{equation*}
f(x_k)-f^*\le \frac{\theta}
{\left(1+\frac{1}{1+\nu}\theta^{\frac{1}{\nu}}(L_{\nu}\delta^{1+\nu})^{-\frac{1}{\nu}}\cdot k\right)^{\nu}}
=O\left(\frac{1}{k^{\nu}}\right).
\end{equation*}
\item
If If $f'$ is Lipschitz continuous with constant $L$, then
\begin{equation*}
f(x_k)-f^*\le \frac{\theta}
{1+\frac{\theta}{2L\delta^2}\cdot k}
=O\left(\frac{1}{k}\right).
\end{equation*}
\end{itemize}
Here $\delta={\rm diam}(C)$.

\end{theorem}

\begin{proof}
First observe from (\ref{vi:1}) that $\langle f'(x_k),x_k-\bar{x}_k\rangle\ge 0$.
By (\ref{curvature-order2}), we have
\begin{align*}
f(x_{k+1})&=\min_{0\le \gamma\le 1} f(x_k+\gamma(\bar{x}_k-x_k))\\
&\le \min_{0\le \gamma\le 1}\{f(x_k)+\gamma\langle f'(x_k),\bar{x}_k-x_k\rangle+\frac{\gamma^\sigma}{\sigma}C^{(\sigma)}_f\}
=:\varphi_k(\gamma).
\end{align*}
The minimizer $\bar{\gamma}\in\mathbb{R}$ satisfies the first optimality condition:
$$\langle f'(x_k),\bar{x}_k-x_k\rangle+\gamma^{\sigma-1}C^{(\sigma)}_f=0,\quad {\rm that\, is,}\quad
\bar{\gamma}^{\sigma-1}=\frac{\langle f'(x_k),x_k-\bar{x}_k\rangle}{C^{(\sigma)}_f}\ge 0.$$
If $\bar{\gamma}\le 1$, then it follows that
\begin{align*}
f(x_{k+1})
&\le f(x_k)+\bar{\gamma}\langle f'(x_k),\bar{x}_k-x_k\rangle+\frac{\bar{\gamma}^\sigma}{\sigma}C^{(\sigma)}_f\\
&=f(x_k)-\bar{\gamma}[\langle f'(x_k),x_k-\bar{x}_k\rangle-\frac{1}{\sigma}\bar{\gamma}^{\sigma-1}C^{(\sigma)}_f]\\
&=f(x_k)-(1-\frac{1}{\sigma})\bar{\gamma}\langle f'(x_k),x_k-\bar{x}_k\rangle\\
&=f(x_k)-\frac{\sigma-1}{\sigma (C^{\left(\sigma\right)}_f)^{\frac{1}{\sigma-1}}}\langle f'(x_k),x_k-\bar{x}_k\rangle^{\frac{\sigma}{\sigma-1}}.
\end{align*}

It turns out that
\begin{equation}\label{ineq:1}
\langle f'(x_k),x_k-\bar{x}_k\rangle\le \left(\frac{\sigma}{\sigma-1}\right)^{\frac{\sigma-1}{\sigma}}
  (C^{(\sigma)}_f)^{\frac1{\sigma}}[f(x_k)-f(x_{k+1})]^{\frac{\sigma-1}{\sigma}}.
\end{equation}

If $\bar{\gamma}>1$, then
$$\langle f'(x_k),x_k-\bar{x}_k\rangle>C^{(\sigma)}_f$$
which then implies that
\begin{align*}
f(x_{k+1})&\le \varphi_k(1)
=f(x_k)+\langle f'(x_k),\bar{x}_k-x_k\rangle+\frac{1}{\sigma}C^{(\sigma)}_f\\
&\le f(x_k)-(1-\frac{1}{\sigma})\langle f'(x_k),x_k-\bar{x}_k\rangle.
\end{align*}

It turns out that
\begin{equation}\label{ineq:2}
\langle f'(x_k),x_k-\bar{x}_k\rangle\le \frac{\sigma}{\sigma-1}[f(x_k)-f(x_{k+1})].
\end{equation}
Combining (\ref{ineq:1}) and (\ref{ineq:2}) yields
\begin{align}\label{ineq:3}
0\le\langle f'(x_k),x_k-\bar{x}_k\rangle
\le\max\bigg\{&\left(\frac{\sigma}{\sigma-1}\right)^{\frac{\sigma-1}{\sigma}}
(C^{(\sigma)}_f)^{\frac1{\sigma}}[f(x_k)-f(x_{k+1})]^{\frac{\sigma-1}{\sigma}},\nonumber \\
&\quad \frac{\sigma}{\sigma-1}[f(x_k)-f(x_{k+1})]\bigg\}
\end{align}
for all $k$.

Now since $f(x_k)-f(x_{k+1})\to 0$ as $k\to\infty$, we may assume that
$$ \left(\frac{\sigma}{\sigma-1}\right)^{\frac{\sigma-1}{\sigma}}
  (C^{(\sigma)}_f)^{\frac1{\sigma}}[f(x_k)-f(x_{k+1})]^{\frac{\sigma-1}{\sigma}}
  >\frac{\sigma}{\sigma-1}[f(x_k)-f(x_{k+1})]$$
for all $k$; consequently from (\ref{ineq:3}) we get
\begin{equation}\label{f(x-k)}
0\le\langle f'(x_k),x_k-\bar{x}_k\rangle\le
\left(\frac{\sigma}{\sigma-1}\right)^{\frac{\sigma-1}{\sigma}}
  (C^{(\sigma)}_f)^{\frac1{\sigma}}[f(x_k)-f(x_{k+1})]^{\frac{\sigma-1}{\sigma}}.
\end{equation}
This can be rewritten as
\begin{equation*}
f(x_{k+1})\le f(x_k)-\frac{\sigma-1}{\sigma}(C^{(\sigma)}_f)^{-\frac{1}{\sigma-1}}
\langle f'(x_k),x_k-\bar{x}_k\rangle^{\frac{\sigma}{\sigma-1}}
\end{equation*}
which, together with (\ref{f'xk0}), can be further rewritten as
\begin{equation*}
\theta_{k+1}\le \theta_k
-\frac{\sigma-1}{\sigma}(C^{(\sigma)}_f)^{-\frac{1}{\sigma-1}}\theta_k^{\frac{\sigma}{\sigma-1}}.
\end{equation*}
Consequently, by Lemma \ref{le:2:ineq}, we get
$$\theta_k\le \frac{\theta_0}
{\left(1+\frac{1}{\sigma}\theta_0^{\frac{1}{\sigma-1}}(C^{(\sigma)}_f)^{\frac{1}{1-\sigma}}\cdot k\right)^{\sigma-1}}
=O\left(\frac{1}{k^{\sigma-1}}\right).$$
\end{proof}

%
%

\subsection{Stepsizes by Open Loop Rule}

\begin{theorem}\label{th:sigma}
Let $C$ be a nonempty closed bounded convex subset of a Banach space $X$
and $f: X\to\mathbb{R}$ be a continuously Fr\'echet differentiable, convex function.
Assume there exists $\sigma>1$ such that
the curvature constant of order $\sigma$, $C_f^{(\sigma)}$, is finite.
Let $\{x_k\}$ be generated by the Frank-Wolfe algorithm (\ref{al:FWA:1})
with stepsizes $\{\gamma_k\}\subset (0,1]$ satisfying the open loop conditions (C1) and (C2).
Then
\begin{equation}\label{f-f*}
f(x_k)-f^*\le \frac{\sigma^\sigma\Delta}{k^{\sigma-1}}\quad\mbox{ for all $k\ge 1$},
\end{equation}
where $\Delta=\max\{f(x_0)-f^*,\frac{1}{\sigma}C^{(\sigma)}_f\}$.
In particular, we have
\begin{itemize}
\item
If $f'$ is $\nu$-H\"older continuous, then
\begin{equation*}
f(x_k)-f^*\le O\left(\frac{1}{k^{\nu}}\right).
\end{equation*}
\item
If If $f'$ is Lipschitz continuous, then
\begin{equation*}
f(x_k)-f^*\le O\left(\frac{1}{k}\right).
\end{equation*}
\end{itemize}

\end{theorem}

\begin{proof}
By definition of the curvature constant of order $\sigma$, we get
\begin{align}\label{fxk++}
f(x_{k+1})&=f(x_k+\gamma_k(\bar{x}_k-x_k)) \nonumber \\
&\le f(x_k)-\gamma_k\langle f'(x_k),x_k-\bar{x}_{k}\rangle+\frac{\gamma_k^{\sigma}}{\sigma}C^{(\sigma)}_f.
\end{align}
This together with  (\ref{f'xk0}) implies that
\begin{equation}\label{theta-k+1}
\theta_{k+1}\le (1-\gamma_k)\theta_k+\frac{\gamma_k^{\sigma}}{\sigma}C^{(\sigma)}_f.
\end{equation}
It turns out from Lemma \ref{le:convergence-tool} that $\theta_k\to 0$, that is, $f(x_k)\to f^*$.

Now define $\{\beta_k\}$ by
\begin{equation}\label{eq:gamma-sigma}
\beta_{k+1}=(1-\gamma_k)\beta_k+\gamma_k^\sigma,\quad \beta_0=1.
\end{equation}

\begin{lemma}\label{le:beta-sig}
Let $\{\beta_k\}\subset (0,1]$ be defined  by (\ref{eq:gamma-sigma}). Then
\begin{equation}\label{eq:beta-k}
\beta_k\le\frac{\sigma^{\sigma}}{k^{\sigma-1}},\quad k\ge 1.
\end{equation}
\end{lemma}

\noindent{\em Proof of Lemma \ref{eq:beta-k}.}
Since $\beta_k\le 1$ for all $k$, we need to verify (\ref{eq:beta-k}) for all $k$ such that $\frac{\sigma^{\sigma}}{k^{\sigma-1}}\le 1$,
that is, $k\ge {\sigma}^{\sigma/(\sigma-1)}$.
Set $\xi=\sigma^{\sigma}$. We will prove (\ref{eq:beta-k}) by induction.
Assume (\ref{eq:beta-k}) is valid for some $k>{\sigma}^{\sigma/(\sigma-1)}$ and we will prove (\ref{eq:beta-k}) for $k+1$.
Namely,
\begin{equation}\label{eq:beta-k+1}
\beta_{k+1}\le\frac{\xi}{(1+k)^{\sigma-1}}.
\end{equation}
By (\ref{eq:gamma-sigma}), it suffices to prove that
\begin{equation}\label{eq:beta-k+2}
(1-\gamma_k)\frac{\xi}{k^{\tau}}+\gamma_k^{\tau+1}\le\frac{\xi}{(1+k)^{\tau}}.
\end{equation}
Here $\tau=\sigma-1\in (0,1]$. Consider the function
$$h(\gamma):=(1-\gamma)\frac{\xi}{k^{\tau}}+\gamma^{\tau+1}, \quad 0\le\gamma\le 1.$$
Then, $h(0)=\frac{\xi}{k^{\tau}}<1$, $h(1)=1$, and
$$h'(\gamma)=-\frac{\xi}{k^{\tau}}+(\tau+1)\gamma^{\tau},\quad
h''(\gamma)=\tau(\tau+1)\gamma^{\tau-1}.$$
Thus, $h$ is a strictly convex function of $\gamma>0$, and the unique solution of $h'(\gamma)=0$ is
given by
$$\hat{\gamma}^\tau=\frac{\xi}{k^{\tau}}\cdot\frac1{\tau+1}\quad {\rm or}\quad
\hat{\gamma}=\left(\frac{\xi}{\tau+1}\right)^{\frac1{\tau}}\frac{1}{k}.$$
It turns out that
\begin{align*}
\min_{0\le\gamma\le 1}h(\gamma)&=h(\hat{\gamma})=(1-\hat{\gamma})\frac{\xi}{k^{\tau}}+\hat{\gamma}^{\tau+1}\\
&=(1-\hat{\gamma})\frac{\xi}{k^{\tau}}+\hat{\gamma}\cdot\frac{\xi}{k^{\tau}}\cdot\frac1{\tau+1}\\
&=\frac{\xi}{k^{\tau}}\left(1-\hat{\gamma}\frac{\tau}{\tau+1}\right)\\
&=\frac{\xi}{k^{\tau}}\left(1-\frac{\tau}{\tau+1}\left(\frac{\xi}{\tau+1}\right)^{\frac1{\tau}}\frac{1}{k}\right).
\end{align*}
We claim that (assuming $\xi\ge(\tau+1)^{\tau+1}=\sigma^{\sigma}$)
\begin{equation}\label{eq:xi}
\frac{\xi}{(1+k)^{\tau}}>h(\hat{\gamma})=
\frac{\xi}{k^{\tau}}\left(1-\frac{\tau}{\tau+1}\left(\frac{\xi}{\tau+1}\right)^{\frac1{\tau}}\frac{1}{k}\right).
\end{equation}
As a matter of fact, setting $x=\frac1{k}$, we equivalently reduce (\ref{eq:xi}) to
\begin{equation}\label{eq:xii}
\frac{1}{(1+x)^{\tau}}>1-ax,
\end{equation}
with $a=\frac{\tau}{\tau+1}\left(\frac{\xi}{\tau+1}\right)^{\frac1{\tau}}$.
Now consider the function:
$$g(x):=(1+x)^\tau(1-ax),\quad 0<x<1.$$
It is easy to find that
$$g'(x)=(1+x)^{\tau-1}[(\tau-a)-a(1+\tau)x]<0$$
for all $x\in (0,1)$ since $\tau\le a$ for $\xi\ge(\tau+1)^{\tau+1}$.
This shows that $g$ is decreasing; consequently, $g(x)<g(0)=1$,
which proves (\ref{eq:xii}) and hence (\ref{eq:xi}).

Next we continue the proof of Theorem \ref{th:sigma} by setting $\Delta=\max\{\theta_0,\frac{1}{\sigma}C^{(\sigma)}_f\}$.
We can easily prove by induction and using (\ref{theta-k+1}) that
$$\theta_k\le\Delta\beta_k$$
for all $k\ge 1$.
By Lemma \ref{le:beta-sig}, we get $\theta_k\le\frac{\Delta\xi}{k^{\sigma-1}}$.
This is (\ref{f-f*}) and Theorem \ref{th:sigma} is proved.
\end{proof}

\begin{remark}
In \cite[Theorem 3]{DH1978}, Dunn and Harshbarger assumed Lipschitz continuity of $f'$ and the condition
on the stepsize sequence $\{\gamma_k\}$:
\begin{equation*}
\gamma_{k+1}=\gamma_k-\frac12\gamma_k^2,\quad \gamma_0=1.
\end{equation*}
We remark that this condition implies $\frac{1}{k+1}\le\gamma_k\le\frac{2}{k+1}$ for all $k\ge 0$, hence
the open loop conditions (C1) and (C2).

\end{remark}
\begin{remark}
The use of H\"older continuous gradient in optimization appeared in Nesterov's recent work \cite{BNV2012,Nesterov2015}.
\end{remark}

\begin{remark}
The convergence rate $O(\frac1{k})$ of FWA (\ref{al:FWA:1}) (with Lipschitz Fr\'echet continuous derivative) can't be improved 
even for strongly convex objective functions, as shown by the following example provided to me by R. Polyak (private communication).
\begin{eg}
Consider, in $\mathbb{R}^n$, the minimization problem $\min_{x\in C} f(x)$, where $f(x)=\frac12\|x\|_2^2$ and 
$C=\{x=(x_1,\cdots,x_n)^\top\in\mathbb{R}^n: x_i\ge 0 \  (\forall i), \ \sum_{i=1}^n x_i=1\}$.
Then the unique optimal solution $x^*$ is given by $x_i^*=\frac1{n}$ for all $1\le i\le n$, and the 
optimal value is $f(x^*)=\frac1{2n}$. 
Let $\{x_k\}$ be generated by FWA (\ref{al:FWA:1}) with $\gamma_k=\frac{2}{k+2}$. 
If the initial guess $x_0\not=x^*$ and if $k\le \frac{n}2-1$, then 
there holds the lower bound:
$$f(x_k)-f(x^*)\ge\frac1{4(k+1)}.$$
\end{eg}
\end{remark}


\section{Generalized Frank-Wolfe Algorithm}
\label{sec:gFWA}

This section is devoted to an extension of FWA to an algorithm,
which is referred to as a generalized Frank-Wolfe algorithm (gFWA), for solving the composite
optimization problem which is recalled below:
\begin{equation}\label{min:f+g}
\min_{x\in C} \varphi(x):=f(x)+g(x),
\end{equation}
where $X$ is a Banach space, $C$ is a closed bounded convex subset of $X$, and
$f, g\in\Gamma_0(X)$, namely, $f, g: X\to (-\infty,\infty]:=\overline{\mathbb{R}}$ are
proper, lower semicontinuous, and convex functions.

We shall use $S$ to denote the set of solutions of (\ref{min:f+g}) and assume $S$ is nonempty.

Furthermore, we always assume that $f$ is continuously Fr\'echet differentiable, and
$C\subset {\rm dom}(g):=\{x\in X: g(x)<\infty\}$.
The Frank-Wolfe algorithm applied to (\ref{min:f+g}) is referred to as a generalized Frank-Wolfe algorithm
(gFWA) which was first considered in \cite{BBLM07,BLM} in some special cases and
which generates a sequence $\{x_k\}_{k=0}^\infty$, with $x_0\in C$ arbitrary, via the iteration procedure:
\begin{subequations} \label{al:gFWA}
  \begin{numcases}{\hbox{}}
\label{gFWA1}
 \bar{x}_k=\arg\min_{x\in C} \langle f'(x_k),x\rangle+g(x),  \\[0.05cm]
 \label{gFWA2}
x_{k+1}=x_k+\gamma_k(\bar{x}_k-x_k).
 \end{numcases}
\end{subequations}
Here $\gamma_k\in [0,1)$ is the stepsize of the $k$th iteration.
In (\ref{gFWA1}) we actually use the first-order linear approximation to $f(x)$ at $x_k$.
Note that (\ref{gFWA1}) is a convex minimization problem.

Similarly to FWA, the convergence and rate of convergence of gFWA depend on the choice of the
stepsizes $\{\gamma_k\}$.

\subsection{Stepsize by Line Minimization}

\begin{theorem}\label{th:gFWA}
Assume the Fr\'echet derivative $f'$ of $f$ is uniformly continuous on $C$.
Let $\{x_k\}$ be generated by the generalized FWA (\ref{al:gFWA}),
where the sequence of stepsizes, $\{\gamma_k\}$, is selected by the
line minimization search method:
\begin{equation}\label{eq:4:gamma}
\gamma_k=\arg\min_{\gamma\in [0,1]} \{f(x_k+\gamma(\bar{x}_k-x_k))+g(x_k+\gamma(\bar{x}_k-x_k))\}.
\end{equation}
Assume the subproblem (\ref{gFWA1}) is solvable for each $k$.
Then we have:
\begin{itemize}
\item[(i)]
$\varphi(x_{k+1})\le \varphi(x_k)$ for all $k$, and
\item[(ii)]
$\lim_{k\to\infty} \varphi(x_k)=\varphi^*:=\min_{x\in C}\varphi(x)$.
\end{itemize}

\end{theorem}

\begin{proof}

The optimality condition for (\ref{gFWA1}) gives that
\begin{equation}\label{eq:4:g(xk)}
-f'(x_k)\in\partial g(\bar{x}_k).
\end{equation}
Then the subdifferential inequality yields:
\begin{equation}\label{eq:4:g(xk)}
\langle f'(x_k),x-\bar{x}_k\rangle\ge g(\bar{x}_k)-g(x)
\end{equation}
for all $x\in X$.

Take a null sequence $\{\tau_k\}$ in $(0,1)$ such that $\sum_{k=0}^\infty\tau_k=\infty$
and define a function $\varphi_k$ by
$$\varphi_k(\gamma):=f(x_k+\gamma(\bar{x}_k-x_k))+g(x_k+\gamma(\bar{x}_k-x_k)),\quad 0\le \gamma\le 1.$$
Then the stepsize $\gamma_k$ is given by
$$\gamma_k=\arg\min \{\varphi_k(\gamma): 0\le \gamma\le 1\}.$$
It immediately turns out that
$$\varphi(x_{k+1})=\varphi_k(\gamma_k)\le\varphi(0)=\varphi(x_k).$$
Namely, $\{\varphi(x_k)\}$ is decreasing, hence $\lim_{k\to\infty}\varphi(x_k)$ exists.
We also have
\begin{align}\label{f(x}
f(x_k+\tau_k(\bar{x}_k-x_k))
&=f(x_k)+\tau_k\langle f'(x_k),\bar{x}_k-x_k)\rangle \nonumber \\
&\quad +\tau_k\int_0^1 \langle f'(x_k+t\tau_k(\bar{x}_k-x_k))-f'(x_k),\bar{x}_k-x_k\rangle dt \nonumber\\
&\le f(x_k)+\tau_k\langle f'(x_k),\bar{x}_k-x_k)\rangle+\tau_k\varepsilon_k.
\end{align}
Here,
$$\varepsilon_k=\delta\cdot\sup_{0\le t\le 1}\|f'(x_k+t\tau_k(\bar{x}_k-x_k))-f'(x_k)\|,\quad
\delta:={\rm diam}(C).$$
Take $x^*\in S$ and set $\theta_k=\varphi(x_k)-\varphi(x^*)=f(x_k)-f(x^*)+g(x_k)-g(x^*)$.

From (\ref{gFWA1}) or (\ref{eq:4:g(xk)}), it follows that
$$\langle f'(x_k),x\rangle+g(x)\ge \langle f'(x_k),\bar{x}_k\rangle+g(\bar{x}_k).$$
This results in that
\begin{align*}
f(x)&\ge f(x_k)+\langle f'(x_k),x-x_k\rangle\\
&\ge f(x_k)+\langle f'(x_k),\bar{x}_k-x_k\rangle+g(\bar{x}_k)-g(x).
\end{align*}
In particular, we obtain by taking $x:=x^*\in S$
\begin{equation}\label{eq:3:f1}
\langle f'(x_k),\bar{x}_k-x_k\rangle\le f(x^*)+g(x^*)-f(x_k)-g(\bar{x}_k).
\end{equation}
Using (\ref{f(x}), (\ref{eq:3:f1}) and convexity of $g$, we are able to estimate $\varphi(x_{k+1})$ as follows.
\begin{align}\label{eq:3:phi1}
\varphi(x_{k+1})&\le\varphi_k(\tau_k)=f(x_k+\tau_k(\bar{x}_k-x_k))+g(x_k+\tau_k(\bar{x}_k-x_k)) \nonumber \\
&\le f(x_k)+\tau_k\langle f'(x_k),\bar{x}_k-x_k)\rangle+\tau_k\varepsilon_k
+g(x_k)+\tau_k[g(\bar{x}_k)-g(x_k)] \nonumber \\
&\le f(x_k)+\tau_k[f(x^*)+g(x^*)-f(x_k)-g(\bar{x}_k)]+\tau_k\varepsilon_k \nonumber \\
&\quad +g(x_k)+\tau_k[g(\bar{x}_k)-g(x_k)] \nonumber \\
&=f(x_k)+g(x_k)+\tau_k[f(x^*)+g(x^*)-f(x_k)-g(x_k)]+\tau_k\varepsilon_k.
\end{align}
Subtracting by $\varphi(x^*)=f(x^*)+g(x^*)$ from both sides of (\ref{eq:3:phi1}) yields
\begin{equation}\label{eq:3:theta}
\theta_{k+1}\le (1-\tau_k)\theta_k+\tau_k\varepsilon_k.
\end{equation}
Now applying Lemma \ref{le:convergence-tool} to (\ref{eq:3:theta}), we conclude that $\theta_k\to 0$; namely,
$\varphi(x_k)\to \varphi(x^*)$. This completes the proof.
\end{proof}

\subsection{Stepsize by Open Loop Rule}

Consider the generalized Frank-Wolfe algorithm (\ref{al:gFWA}) where the sequence of stepsizes $\{\gamma_k\}$ is
selected by the open loop rule (C1)-(C2) of Theorem \ref{th:open-loop-1}.

\begin{theorem}\label{th:gFWA-OpenLoop}
Consider the sequence $\{x_k\}$ generated by the generalized Frank-Wolfe algorithm (\ref{al:gFWA}).
Assume the conditions below are satisfied:
\begin{enumerate}
\item[{\rm (i)}]
the Fr\'echet derivative $f'$ is uniformly continuous over $C$;
\item[{\rm (ii)}]
the stepsizes $\{\gamma_k\}\subset (0,1]$ satisfy the open loop conditions:
\begin{itemize}
\item[(C1)]
$\lim_{k\to\infty}\gamma_k=0$,
\item[(C2)] $\sum_{k=0}^\infty\gamma_k=\infty$.
\end{itemize}
\end{enumerate}
Then $\lim_{k\to\infty}\varphi(x_k)=\varphi^*:=\inf_C\varphi$.
\end{theorem}
\begin{proof}
The proof is some minor alterations of that of Theorem \ref{th:gFWA}. As before, we set
$\theta_k=\varphi(x_k)-\varphi^*$. It is easily seen that we still have (\ref{f(x}) where
$\tau_k$ is replaced with $\gamma_k$. Namely, we have
\begin{equation}\label{eq:3:f2}
f(x_k+\gamma_k(\bar{x}_k-x_k))
\le f(x_k)+\gamma_k\langle f'(x_k),\bar{x}_k-x_k)\rangle+\gamma_k\varepsilon_k.
\end{equation}
Here
$$\varepsilon_k=\delta\cdot\sup_{0\le t\le 1}\|f'(x_k+t\gamma_k(\bar{x}_k-x_k))-f'(x_k)\|,\quad
\delta:=\sup\{\|\bar{x}_k-x_k\|: k\ge 0\}.$$
Note that we have $\varepsilon_k\to 0$ since $f'$ is uniformly continuous on $C$.

Observe that (\ref{eq:3:f1}) remains valid; observe also that (\ref{eq:3:phi1}) remains valid with $\tau_k$ substituted by
$\gamma_k$ for each $k$; that is,
\begin{align}\label{eq:3:phi2}
\varphi(x_{k+1})&=f(x_k+\gamma_k(\bar{x}_k-x_k))+g(x_k+\gamma_k(\bar{x}_k-x_k)) \nonumber \\
&\le f(x_k)+g(x_k)+\gamma_k[f(x^*)+g(x^*)-f(x_k)-g(x_k)]+\gamma_k\varepsilon_k.
\end{align}
By subtracting $\varphi(x^*)$ from both sides of (\ref{eq:3:phi2}), we get
\begin{equation}\label{eq:3:theta2}
\theta_{k+1}\le (1-\gamma_k)\theta_k+\gamma_k\varepsilon_k.
\end{equation}
So again applying Lemma \ref{le:convergence-tool}, we conclude that $\theta_k\to 0$
and this finishes the proof.
\end{proof}

\subsection{Convergence of Iterates}

The convergence results on iterates of FWA can partially be extended to gFWA.

\begin{theorem}\label{th:gFWA-iterates}
Let $X$ be a reflexive Banach space and consider
the the generalized FWA (\ref{al:gFWA}) with the stepsize sequence $\{\gamma_k\}$ selected by either the line minimization
search method in Theorem \ref{th:FWA-convergence1} or the
open loop rule (C1)-(C2) of Theorem \ref{th:open-loop-1}.
\begin{enumerate}
\item[{\rm (i)}]
If $f$ is strictly convex, then  $\{x_k\}$ converges weakly to the unique solution of (\ref{min:f+g}).
\item[{\rm (ii)}]
If $f$ is uniformly convex, then $\{x_k\}$ converges in norm to the unique solution of (\ref{min:f+g}).
\item[{\rm (iii)}]
If $C$ is compact in the norm topology, if the stepsizes $\{\gamma_k\}$ is selected by the open loop rule,
and if $\{x_k\}$ has at most finitely many cluster points, then
$\{x_k\}$ converges in norm to a solution of (\ref{min:f+g}).
\end{enumerate}

\end{theorem}
\begin{proof}
(i)
The strict convexity of $f$ implies the strict convexity of $\varphi$, hence (\ref{min:f+g})
has a unique solution, and we write the optimal solution set $S=\{x^*\}$.

Now Theorems \ref{th:gFWA} and \ref{th:gFWA-OpenLoop} assure that $x^*$ is the only weak
cluster point of $\{x_k\}$. This is equivalent to fact that $\{x_k\}$ converges weakly to $x^*$.

(ii)
First observe by (i) that $\{x_k\}$ is weakly convergent to the unique solution $x^*$ of (\ref{min:f+g}).
Now let $\delta$ be a modulus of convexity of $f$ (i.e., Eq. (\ref{eq:h}) holds for $f$).
This actually implies that $\varphi=f+g$ is also uniformly convex with the same modulus $\delta$
of convexity.  Moreover, it is not hard to find that (\ref{eq:3:f3}) turns out to be
\begin{equation}\label{eq:3:phi}
\varphi(x)\ge \varphi(y)+\langle \varphi'(y),x-y\rangle+\delta(\|x-y\|)
\end{equation}
for all $x,y\in X$, where $\varphi'(y)=f'(y)+g'(y)$. [Here $g'(y)\in\partial g(y)$ is
a subgradient of $g$ at $y$.]

Taking $x:=x_k$ and $y:=x^*\in S$ in (\ref{eq:3:phi}) yields that
\begin{equation}\label{eq:3:phi2}
\varphi(x_k)\ge \varphi(x^*)+\langle \varphi'(x^*),x_k-x^*\rangle+\delta(\|x_k-x^*\|).
\end{equation}
Since $\varphi(x_k)\to \varphi(x^*)$ and $x_k\to x^*$ weakly (hence $\langle\varphi'(x^*),x_k-x^*\rangle\to 0$),
taking the limit in (\ref{eq:3:phi2}) as $k\to\infty$,
we immediately get $\delta(\|x_k-x^*\|)\to 0$. Consequently, $x_k\to x^*$ in norm.

(iii)
This is exactly the same as the proof of part (iv) of Theorem \ref{th:FWA-iterates}
since we still have the fact that $\|x_{k+1}-x_k\|\to 0$ which follows from (\ref{gFWA2}),
the assumption $\gamma_k\to 0$, and the boundedness of $\{x_k\}\cup\{\bar{x}_k\}$.
\end{proof}

\subsection{Rate of Convergence}

In this section we discuss the rate of convergence of the generalized FWA (\ref{al:gFWA})
and again distinguish two cases of the ways of choosing the stepsizes $\{\gamma_k\}$.

\subsubsection{Stepsizes by the Line Minimization Search Method}

\begin{theorem}\label{th:rate-gFWA1}
Let $\{x_k\}$ be generated by the generalized FWA (\ref{al:gFWA}),
where the sequence of stepsizes, $\{\gamma_k\}$, is selected by the
line minimization search method:
\begin{equation}\label{eq:4:gamma}
\gamma_k=\arg\min_{\gamma\in [0,1]} \{f(x_k+\gamma(\bar{x}_k-x_k))+g(x_k+\gamma(\bar{x}_k-x_k))\}.
\end{equation}
Assume there exists $\sigma>1$ such that the curvature constant of order $\sigma$ of $f$ over $C$,
$C_f^{(\sigma)}$ defined by (\ref{curvature-order1}), is finite.
Then we have, for $k\ge 1$,
\begin{equation}
\varphi(x_k)-\varphi^*\le \frac{\theta_0}
{\left(1+\frac{1}{\sigma}\theta_0^{\frac{1}{\sigma-1}}(C^{(\sigma)}_f)^{\frac{1}{1-\sigma}}\cdot k\right)^{\sigma-1}}
=O\left(\frac{1}{k^{\sigma-1}}\right),
\end{equation}
where $\theta_0=\varphi(x_0)-\varphi^*$.

In particular, we have
\begin{itemize}
\item
If $f'$ is $\nu$-H\"older continuous with constant $L_\nu$, then
\begin{equation*}
f(x_k)-f^*\le \frac{\theta_0}
{\left(1+\frac{1}{1+\nu}\theta_0^{\frac{1}{\nu}}(L_{\nu}\delta^{1+\nu})^{-\frac{1}{\nu}}\cdot k\right)^{\nu}}
=O\left(\frac{1}{k^{\nu}}\right).
\end{equation*}
\item
If If $f'$ is Lipschitz continuous with constant $L$, then
\begin{equation*}
f(x_k)-f^*\le \frac{\theta_0}
{1+\frac{\theta_0}{2L\delta^2}\cdot k}
=O\left(\frac{1}{k}\right).
\end{equation*}
\end{itemize}
Here $\delta={\rm diam}(C)$.

\end{theorem}

\begin{proof}
The proof given below is appropriate adaptations of the proof of Theorem \ref{th:FWA-rate1}.
We use the same notation of the errors: $\theta_k:=\varphi(x_k)-\varphi(x^*)=f(x_k)+g(x_k)-f(x^*)-g(x^*)$
with $x^*\in S$. Using (\ref{curvature-order1}) and the convexity of $g$, we deduce that
\begin{align*}
\varphi(x_{k+1})&=\min_{\gamma\in [0,1]} f(x_k+\gamma(\bar{x}_k-x_k))+g(x_k+\gamma(\bar{x}_k-x_k))\\
&\le \min_{\gamma\in [0,1]} \varphi(x_k)+\gamma\langle f'(x_k),\bar{x}_k-x_k\rangle+\frac{\gamma^\sigma}{\sigma}C_f^{(\sigma)}
 +\gamma(g(\bar{x}_k)-g(x_k))=:\varphi_k(\gamma).
\end{align*}
Solving $\bar{\gamma}$ from the equation:
$$\varphi_k'(\gamma)=\langle f'(x_k),\bar{x}_k-x_k\rangle+\gamma^{\sigma-1}C_f^{(\sigma)}+g(\bar{x}_k)-g(x_k)=0,$$
we get
\begin{equation}\label{eq:gamma^sigma}
\bar{\gamma}^{\sigma-1}=\frac{1}{C_f^{(\sigma)}}[\langle f'(x_k),x_k-\bar{x}_k\rangle
+g(x_k)-g(\bar{x}_k)]\ge 0.
\end{equation}
Notice that
$$\varphi_k(0)=\varphi(x_k),\quad \varphi_k(1)=\varphi(x_k)+\frac{1}{\sigma}C_f^{(\sigma)}
+\langle f'(x_k),\bar{x}_k-x_k\rangle+g(\bar{x}_k)-g(x_k).$$
If $\bar{\gamma}\le 1$, then using (\ref{eq:gamma^sigma}) we have
\begin{align*}
\varphi(x_{k+1})&\le\varphi_k(\bar{\gamma})
=\varphi(x_k)+\bar{\gamma}\left(\langle f'(x_k),\bar{x}_k-x_k\rangle+g(\bar{x}_k)-g(x_k)
  +\frac{\bar{\gamma}^{\sigma-1}}{\sigma}C_f^{(\sigma)}\right)\\
&=\varphi(x_k)+(1-\frac1{\sigma})\bar{\gamma}\left(\langle f'(x_k),\bar{x}_k-x_k\rangle+g(\bar{x}_k)-g(x_k)
  \right)\\
&=\varphi(x_k)+(1-\frac1{\sigma})\frac{1}{\left(C_f^{(\sigma)}\right)^{\frac{1}{\sigma-1}}}
\left(\langle f'(x_k),\bar{x}_k-x_k\rangle+g(\bar{x}_k)-g(x_k)\right).
\end{align*}
It turns out that
\begin{equation} \label{eq:f'g1}
\langle f'(x_k),x_k-\bar{x}_k\rangle+g(x_k)-g(\bar{x}_k)
\le\left(\frac{\sigma}{\sigma-1}\right)^{\frac{\sigma-1}{\sigma}}\left(C_f^{(\sigma)}\right)^{\frac{1}{\sigma}}
\left(\varphi(x_k)-\varphi(x_{k+1})\right)^{\frac{\sigma-1}{\sigma}}.
\end{equation}
If $\bar{\gamma}>1$, then from (\ref{eq:gamma^sigma}) we get
$$\langle f'(x_k),x_k-\bar{x}_k\rangle+g(x_k)-g(\bar{x}_k)>C_f^{(\sigma)}.$$
It follows that
\begin{align*}
\varphi(x_{k+1})&\le\varphi_k(1)=\varphi(x_k)
+\langle f'(x_k),\bar{x}_k-x_k\rangle+g(\bar{x}_k)-g(x_k)+\frac{1}{\sigma}C_f^{(\sigma)} \\
&\le \varphi(x_k)
-(1-\frac{1}{\sigma})\langle f'(x_k),x_k-\bar{x}_k\rangle+g(x_k)-g(\bar{x}_k).
\end{align*}
Consequently,
\begin{equation}\label{eq:f'g2}
\langle f'(x_k),x_k-\bar{x}_k\rangle+g(x_k)-g(\bar{x}_k)\le \frac{\sigma}{\sigma-1}(\varphi(x_k)-\varphi(x_{k+1})).
\end{equation}
Combining (\ref{eq:f'g1}) and (\ref{eq:f'g2}) we obtain
\begin{align}\label{eq:f'g3}
&\langle f'(x_k),x_k-\bar{x}_k\rangle+g(x_k)-g(\bar{x}_k) \nonumber\\
&\le \max\bigg\{\left(\frac{\sigma}{\sigma-1}\right)^{\frac{\sigma-1}{\sigma}}\left(C_f^{(\sigma)}\right)^{\frac{1}{\sigma}}
\left(\varphi(x_k)-\varphi(x_{k+1})\right)^{\frac{\sigma-1}{\sigma}},
\frac{\sigma}{\sigma-1}(\varphi(x_k)-\varphi(x_{k+1}))\bigg\}.
\end{align}
Since $\varphi(x_k)-\varphi(x_{k+1})\to 0$, there holds the relation:
\begin{equation}\label{eq:f'g4}
\left(\frac{\sigma}{\sigma-1}\right)^{\frac{\sigma-1}{\sigma}}\left(C_f^{(\sigma)}\right)^{\frac{1}{\sigma}}
\left(\varphi(x_k)-\varphi(x_{k+1})\right)^{\frac{\sigma-1}{\sigma}}>
\frac{\sigma}{\sigma-1}(\varphi(x_k)-\varphi(x_{k+1}))
\end{equation}
for all $k$ big enough. Therefore, we may assume that (\ref{eq:f'g4}) holds for all $k$, from which and  (\ref{eq:f'g3})
we get
\begin{equation}\label{eq:f'g5}
\varphi(x_{k+1})\le\varphi(x_k)-\frac{\sigma-1}{\sigma}\left(C_f^{(\sigma)}\right)^{-\frac{1}{\sigma-1}}
(\langle f'(x_k),x_k-\bar{x}_k\rangle+g(x_k)-g(\bar{x}_k))^{\frac{\sigma-1}{\sigma}}.
\end{equation}
Now by the convexity of $f$ and (\ref{eq:4:g(xk)}), we have, for all $x$,
\begin{align*}
f(x)&\ge f(x_k)+\langle f'(x_k),x-x_k\rangle. \\
&\ge f(x_k)+\langle f'(x_k),\bar{x}_k-x_k\rangle+g(\bar{x}_k)-g(x_k).
\end{align*}
It turns out that
\begin{align*}
\langle f'(x_k),x_k-\bar{x}_k\rangle+g(x_k)-g(\bar{x}_k) \ge \varphi(x_k)-\varphi(x).
\end{align*}
In particular,
\begin{align}\label{eq:f'g6}
\langle f'(x_k),x_k-\bar{x}_k\rangle+g(x_k)-g(\bar{x}_k) \ge \varphi(x_k)-\varphi(x^*)=\theta_k.
\end{align}
Substituting (\ref{eq:f'g6}) into (\ref{eq:f'g5}) yields that
\begin{equation}
\theta_{k+1}\le\theta_k-\frac{\sigma-1}{\sigma}\left(C_f^{(\sigma)}\right)^{{-\frac{1}{\sigma-1}}}
\theta_k^{\frac{\sigma}{\sigma-1}}.
\end{equation}
Applying Lemma \ref{le:2:ineq}, we get
$$\theta_k\le \frac{\theta_0}
{\left(1+\frac{1}{\sigma}\theta_0^{\frac{1}{\sigma-1}}(C^{(\sigma)}_f)^{\frac{1}{1-\sigma}}\cdot k\right)^{\sigma-1}}
=O\left(\frac{1}{k^{\sigma-1}}\right).$$
This ends the proof.
\end{proof}

\subsubsection{Stepsizes by Open Loop Rule}

\begin{theorem}\label{th:rate-gFWA2}
Let $C$ be a nonempty closed bounded convex subset of a Banach space $X$
and $f: X\to\mathbb{R}$ be a continuously Fr\'echet differentiable, convex function.
Assume
the curvature constant of order $\sigma$ of $f$, $C_f^{(\sigma)}$, is finite for some $\sigma\in (1,2]$.
Let $\{x_k\}$ be generated by the generalized Frank-Wolfe algorithm (\ref{al:gFWA})
with the stepsize sequence $\{\gamma_k\}\subset (0,1]$ satisfying the open loop conditions (C1) and (C2).
Then, for $k\ge 1$,
\begin{equation}\label{phi-phi*}
\varphi(x_k)-\varphi^*\le \frac{\sigma^\sigma\Delta}{(k+1)^{\sigma-1}}=O\left(\frac1{k^{\sigma-1}}\right),
\end{equation}
where
$\Delta=\max\{\varphi(x_0)-\varphi^*,\frac{1}{\sigma}C^{(\sigma)}_f\}$.
In particular, we have
\begin{itemize}
\item
If $f'$ is $\nu$-H\"older continuous, then
\begin{equation*}
\varphi(x_k)-\varphi^*\le O\left(\frac{1}{k^{\nu}}\right).
\end{equation*}
\item
If If $f'$ is Lipschitz continuous, then
\begin{equation*}
\varphi(x_k)-\varphi^*\le O\left(\frac{1}{k}\right).
\end{equation*}
\end{itemize}
\end{theorem}
\begin{proof}
Again we set $\theta_k:=\varphi(x_k)-\varphi(x^*)=f(x_k)+g(x_k)-f(x^*)-g(x^*)$
with $x^*\in S$. Using (\ref{curvature-order1}) and the convexity of $g$, we deduce that
\begin{align*}
\varphi(x_{k+1})&=f(x_{k+1})+g(x_{k+1})\\
&=f(x_k+\gamma_k(\bar{x}_k-x_k))+g(x_k+\gamma_k(\bar{x}_k-x_k))\\
&\le f(x_k)+\gamma_k\langle f'(x_k),\bar{x}_k-x_k\rangle+\frac{\gamma_k^\sigma}{\sigma}C_f^{(\sigma)}
 +g(x_k)+\gamma_k(g(\bar{x}_k)-g(x_k))\\
&=\varphi(x_k)-\gamma_k[\langle f'(x_k),x_k-\bar{x}_k\rangle+g(x_k)-g(\bar{x}_k)]
+\frac{\gamma_k^\sigma}{\sigma}C_f^{(\sigma)}.
\end{align*}
This together with (\ref{eq:f'g6}) implies
\begin{align*}
\varphi(x_{k+1})&\le \varphi(x_k)-\gamma_k(\varphi(x_k)-\varphi^*)+\frac{\gamma_k^\sigma}{\sigma}C_f^{(\sigma)}.
\end{align*}
Equivalently,
\begin{equation}\label{eq:theta-gFWA}
\theta_{k+1}\le (1-\gamma_k)\theta_k+\frac{\gamma_k^\sigma}{\sigma}C_f^{(\sigma)}.
\end{equation}
This is exactly (\ref{theta-k+1}). The estimate (\ref{phi-phi*}) therefore follows by repeating
the same argument of the proof of Theorem \ref{th:sigma}.
\end{proof}

\section{Conclusion}
\label{sec:conclusion}

We have studied FWA for solving (\ref{min:0:f}) and generalized FWA for (\ref{min:composite})
in the setting of general Banach spaces.
We have proved convergence of FWA and gFWA under two ways of choosing the stepsizes:
Line minimization search and open loop rule, under the condition that
the Fr\'echet derivative $f'$ of $f$ is uniformly continuous over $C$
(continuity of the gradient $\nabla f$ in the finite-dimensional framework).
To get rate of convergence of FWA and gFWA,
we have introduced the notion of curvature constant of order $\sigma\in (1,2]$ over $C$
and then successfully proved the $O\left(\frac{1}{k^{\nu}}\right)$ rate
of FWA and gFWA if $f'$ is $\nu$-H\"older continuous. In particular, FWA and gFWA
have at least sublinear rate $O\left(\frac{1}{k}\right)$ of convergence if $f'$ is Lipschitz continuous.
We have also studied convergence of the iterates $\{x_k\}$ of FWA and gFWA, and proved
that $\{x_k\}$ converges
(i) weakly  to a solution of (\ref{min:0:f}) and (\ref{min:composite}) if $f$ is strictly convex;
(ii) strongly to a solution of (\ref{min:0:f}) and (\ref{min:composite}) if $f$ is uniformly convex;
and (iii) strongly to a solution of (\ref{min:0:f}) and (\ref{min:composite}) if
$C$ is compact in the norm topology, the stepsizes $\{\gamma_k\}$ are selected by the open loop rule,
and $\{x_k\}$ has at most finitely many cluster points.

Since FWA and gFWA have a sublinear rate of convergence in the case where
$f'$ is Lipschitz continuous, it is an interesting problem of speeding up
the convergence rate of FWA and gFWA using Nesterov's acceleration method \cite{Ne1983,BT2009}.

A summary of the results obtained in this paper is as follows:
\begin{itemize}
\item
Uniform continuity (continuity in finite-dimensional spaces) of $f'$ on $C$
is sufficient to guarantee convergence of FWA and gFWA.

\item
Finite curvature constant of order $\sigma\in (1,2]$ of $f$ over $C$, in particular, H\"older or Lipschitz continuity of $f'$,
guarantees convergence rate of $O\left(\frac{1}{k^\tau}\right)$ of FWA and gFWA, where $\tau\in (0,1]$.

\item
Convergence of the iterates of FWA and gFWA remain more and further investigations.

\end{itemize}


\begin{thebibliography}{10}

\bibitem{BNV2012}
F. Babonneau, Yu. Nesterov, and J.-P. Vial,
Design and operating of gas transmission networks, in:
Operations Research, 2012, pp. 1-14.

\bibitem{BT2009}
A. Beck and M. Teboulle,  A fast iterative shrinkage-thresholding algorithm for linear inverse problems,
SIAM J. Imaging Sci. 2 (2009), no. 1, 183-202.

\bibitem{BBLM07}
T. Bonesky, K. Bredies, D. A. Lorenz and P. Maass,
A generalized conditional gradient method for nonlinear operator equations with sparsity constraints,
Inverse Problems 23 (2007), 2041-2058.

\bibitem{BLM}
K. Bredies, D. A. Lorenz and P. Maass,
A generalized conditional gradient method and its connection to an iterative shrinkage method,
Comput. Optim. Appl. 42 (2009), 173-193.

\bibitem{Ci1990}
J. Cioranescu, Geometry of Banach Spaces-Duality Mappings and Nonlinear Problems, Mathematics and Its
Applications 62, Kluwer Academic Publisher, Dordrecht, The Netherlands, 1990.

\bibitem{DH1978}
J. C. Dunn and S. Harshbarger, Conditional gradient algorithms with open loop step size rules,
J. Math. Anal. Appl. 62 (1978), 432-444.

\bibitem{Frank-Wolfe1956}
M. Frank and P. Wolfe, An algorithm for quadratic programming,
Naval Research Logistics Quarterly  3 (1956), 95-110.

\bibitem{Freund-Grigas2016}
R. M. Freund and P. Grigas, New analysis and results for the Frank-Wolfe method,
Math. Program. A  155 (2016), no.  1, 199-230.

\bibitem{Fu1984}
M. Fukushima,
A modified Frank-Wolfe algorithm for solving the traffic assignment problem,
Transpn. Res. -B 18B (1984), no. 2, 169-177.

\bibitem{Ja2013}
M. Jaggi,
Revisiting Frank-Wolfe: Projection-free sparse convex optimization,
Proceedings of the 30 th International Conference on Machine Learning, Atlanta, Georgia, USA, 2013.
JMLR: W$\&$CP volume 28, 2013. PMLR 28(1): 427-435.


\bibitem{JH2014}
M. Jaggi and Z. Harchaoui,
The recent revival of the Frank-Wolfe algorithm, Optima 95, September 2014, pp. 2-8.

\bibitem{JS2010}
M. Jaggi and M. Sulovsky, A simple algorithm for nuclear norm regularized problems,
ICML-International Conference on Machine Learning, 2010, pp. 471-478.

\bibitem{Ne1983}
Yu. Nesterov, A method for solving the convex programming problem with convergence rate $O(1/k^2)$.
Dokl. Akad. Nauk SSSR 269 (1983), no. 3, 543-547.

\bibitem{Nesterov2015}
Yu. Nesterov,
Universal gradient methods for convex optimization problems,
Math. Program., Ser. A 152 (2015), 381-404.

\bibitem{Nesterov2017}
Yu. Nesterov,
Complexity bounds for primal-dual methods minimizing the model of objective function.
(Preprint.)

\bibitem{Phelps1985}
R. R.  Phelps, Metric projections and the gradient projection method in Banach spaces,
SIAM J. Control and Optim. 26 (1985), no. 6,  973-977.

\bibitem{Polyak1987}
B. T. Polyak, Intrduction to Optimization,
Optimization Software, New York, 1987.

\bibitem{V1987}
D. Van Vliet,
The Frank-Wolfe algorithm for equilibrium traffic assignment viewed as a variational inequality,
Transpn. Res. -B.  21 (1987), no. 1, 87-89.

\bibitem{WLSM2017}
H.-T. Wai, J. Lafond, A. Scaglione, and E. Moulines,
Decentralized Frank-Wolfe algorithm for convexvand non-convex problems,
IEEE Transactions on Automatic Control PP(99):1-1, March 2017.
DOI 10.1109/TAC.2017.2685559

\bibitem{Xu2002}
H. K. Xu, Iterative algorithms for nonlinear operators,
J. London Math. Soc. 66 (2002), no. 2, 240-256.

\bibitem{Xu2011}
H. K. Xu, Averaged mappings and the gradient-projection algorithm,
J. Optim. Theory Appl. 150 (2011), no. 2, 360-378.




\end{thebibliography}
\end{document}